\documentclass[11pt,twoside]{article}
\usepackage{amssymb,mathrsfs,color}
\usepackage{amsfonts,amsmath,latexsym}
\usepackage{enumerate}
\usepackage{authblk}

\usepackage{indentfirst}
\usepackage[amsmath,thmmarks]{ntheorem}
\numberwithin{equation}{section}
\theoremheaderfont{\normalfont\rmfamily\bf}
\theorembodyfont{\normalfont\it }
\newtheorem{theorem}{\indent Theorem}[section]
\newtheorem{lemma}{\indent Lemma} [section]
\newtheorem{corollary}{\indent Corollary} [section]

\newtheorem{definition}{\indent Definition} [section]
\newtheorem{remark}{\indent Remark} [section]

  \theoremstyle{nonumberplain}

    \newenvironment{proof}[1][Proof]{\textbf{#1}~~}{\hfill $\square$ \bigskip}

\allowdisplaybreaks

\def\rr{\mathbb{R}}
\def\rn{{\rr}^n}
\def\loc{{\rm loc}}

\begin{document}

\title{\bf Commutators of the maximal and sharp functions
   with weighted Lipschitz functions
    \footnotetext{* Pu Zhang (Corresponding author), ~Email:  puzhang@sohu.com}
     \footnotetext{~~ Xiaomeng Zhu, ~Email: zhuxiaomeng11@163.com}
      \footnotetext{1 Department of Mathematics, Mudanjiang Normal University, Mudanjiang 157011, China}\\
         \footnotetext{2 Tangshan Qishi Art Middle School, Tangshan 063700, China}
           }

\author{Pu Zhang$^{1,*}$   ~and Xiaomeng Zhu$^{1,2}$ }

\date{ }
\maketitle

\begin{center}\begin{minipage}{14cm}

{\bf Abstract}
Let $M$ be the Hardy-Littlewood maximal function.
Denote by $M_b$ and $[b,M]$ the maximal and the nonlinear
commutators of $M$ with a function $b$.
  The boundedness of $M_b$ and $[b,M]$ on weighted Lebesgue spaces
  are characterized
 when the symbols $b$ belong to weighted Lipschitz (weighted Morrey-Campanato) spaces.
  Some new characterizations for weighted Lipschitz spaces are obtained.
Similar results are also established for the nonlinear commutator
of the sharp function.



\end{minipage}
\end{center}

\section{Introduction and Main Results}

Let $T$ be the  classical singular integral operator. The commutator
$[b,T]$ generated by $T$ and a suitable function $b$ is given by
\begin{equation} \label{equ.1.1}    
[b,T](f)(x) = T\big((b(x)-b)f\big)(x)
       = b(x)T(f)(x)-T(bf)(x).
\end{equation}

A well known result states that $[b,T]$ is bounded on $L^p(\rn)$
for $1<p<\infty$ if and only if $b\in {BMO(\rn)}$
(see Coifman et al. \cite{crw} and Janson \cite{j}).
Another kind of boundedness for $[b,T]$ was given by Janson \cite{j}.
It was proved that $[b,T]$ is bounded from $L^p(\rn)$ to $L^q(\rn)$
for $1<p<q<\infty$ if and only if $b\in {Lip_{\beta}(\rn)}$ with
$0<\beta=n(1/p-1/q)<1$,
where ${Lip_{\beta}(\rn)}$ is the Lipschitz space of order $\beta$.
In 2008, Hu and Gu \cite{hg} considered the  weighted boundedness
of commutator $[b,T]$ when $b$ belongs to weighted Lipschitz spaces.

On the other hand, mapping properties of commutators
generated by the Hardy-Littlewood maximal function and
suitable locally integrable functions have been studied intensively
by many authors in recent years.
 See \cite{agkm, bmr, fg, gma, gul, ms, yyl, zp2017, zp2019, zsw, zw-cmj, zw-mia, zws}
 for instance.
We would like to mention that the regularity and continuity
of commutators of maximal functions were first studied by
Liu, Xue and the first author in \cite{lxz1} and \cite{lxz2}.

As usual, a cube $Q\subset \rn$ always means its sides
parallel to the coordinate axes.
Denote by $|Q|$ the Lebesgue measure and by $\chi_Q$
the characteristic function of $Q$.
 For $f\in {L}_{\loc}^{1}(\rn)$, we write $f_Q={|Q|}^{-1}\int_Q f(x)dx$.
The Hardy-Littlewood maximal function $M$ is defined by
$$M(f)(x)=\sup_{Q\ni x} \frac{1}{|Q|} \int_{Q} |f(y)| dy,
$$
and the sharp function $M^{\sharp}$, introduced by
Fefferman and Stein \cite{fs}, is given by
$$M^{\sharp}(f)(x)=\sup_{Q\ni x}\frac{1}{|Q|} \int_{Q}|f(y)-f_{Q}|{d}y,
$$
where the supremum is taken over all cubes $Q\subset \rn$
containing $x$.

Similar to (\ref{equ.1.1}), we can define two different kinds of
commutators of the Hardy-Littlewood maximal function as follows.

\begin{definition}       \label{def.com} 
Let $b$ be a locally integrable function. The maximal commutator
generated by $M$ and $b$ is given by
$$M_b(f)(x)= M\big((b(x)-b)f\big)(x)
=\sup_{Q\ni x} \frac{1}{|Q|} \int_{Q} |b(x)-b(y)||f(y)|dy,
$$
where the supremum is taken over all cubes $Q\subset\rn$ containing
$x$.

The nonlinear commutator generated by $M$ and $b$ is defined by
$$[b,M](f)(x)=b(x)M(f)(x)-M(bf)(x).
$$
\end{definition}

Similarly, we can also define the nonlinear commutator of
$M^{\sharp}$ and $b$ by
$$[b,M^{\sharp}](f)(x)= bM^{\sharp}(f)(x) -M^{\sharp}(bf)(x).
$$

Operators $M_b$ and $[b,M]$ essentially
differ from each other. Obviously, $M_b$ is positive and
sublinear, but $[b,M]$ and $[b,M^{\sharp}]$ are neither positive nor sublinear.
 The operator $[b,M]$ can
be used in studying the product of a function in $H^1$ and a
function in $BMO$, see \cite{bijz} for instance.

The commutator $[b,M]$ was first studied by Milman and Schonbek \cite{ms}.
In 2000, Bastero, Milman and Ruiz \cite{bmr} studied the necessary
and sufficient condition for the boundedness of $[b,M]$ and
 $[b,M^{\sharp}]$ on $L^p$ spaces. In 2009, Zhang and Wu
\cite{zw2009} considered the same problems for the commutator of
the fractional maximal function. The results were extended to
variable Lebesgue spaces in \cite{zw-cmj} and \cite{zw-mia}.

Recently, Zhang \cite{zp2017} gave some necessary and sufficient conditions
for the boundedness of $M_b$ and $[b,M]$ on Lebesgue and Morrey
spaces when the symbols $b$ belong to Lipschitz space, by which
some new characterizations of Lipschitz functions are given.
Some of the results  
 were extended to variable
exponent Lebesgue spaces in \cite{zp2019}, \cite{zsw} and \cite{yyl},
and to the context of Orlicz spaces in \cite{gul}, \cite{gd}, \cite{gdh}
and \cite{zws}.

In this paper, we will consider the weighted boundedness
for the commutators $M_b$, $[b,M]$ and $[b,M^{\sharp}]$
when the symbols $b$ belong to weighted Lipschitz spaces.
In virtue of the boundedness of the commutators,
 some new characterizations for weighted Lipschitz spaces are given.
  To state the results, we first give some definitions and notations.

A weight will always mean a nonnegative locally integrable function.
As usual, we denote by $A_p$ $(1\le{p}\le \infty)$ the Muckenhoupt
weights classes (see \cite{gr}, \cite{gra} and \cite{stein} for details).
 For a weight $\mu$ and a measurable set $E$ in $\rn$, we write
$$\|f\|_{L^p(\mu)}=\left(\int_{\rn}|f(x)|^p \mu(x) dx\right)^{1/p}~ \mbox{and}~~
 \mu(E)=\int_{E} \mu(x)dx.
$$

Following \cite{g}, we difine the weighted Lipschitz function space.
See also \cite{hg} and \cite{t}.

\begin{definition}     \label{def.wlip}        
Let $1\le{p}\le\infty$, $0<\beta<1$ and $\mu\in{A_{\infty}}$.
The weighted Lipschitz space, denoted by $Lip_{\beta,\mu}^p$, is given by
$$Lip_{\beta,\mu}^p
  =\Big\{f \in{L_{\rm{loc}}^1(\rn)}: \|f\|_{Lip_{\beta ,\mu}^p}< \infty \Big\},
$$
where
$$\|f\|_{Lip_{\beta ,\mu}^p} = \sup_Q \frac{1}{\mu(Q)^{\beta /n}}
 {\bigg(\frac{1}{\mu(Q)} \int_Q |f(x)-f_Q|^p}\mu (x)^{1-p}dx\bigg)^{1/p}.
$$
here and below,
 ``$\sup\limits_Q$'' always means the supremum is taken over all
 cubes $Q\subset \rn$.
\end{definition}

Modulo constants, $Lip_{\beta,\mu}^p$ is a Banach space
with respect to the norm $\|\cdot\|_{Lip_{\beta ,\mu}^p}$.
We simply write $Lip_{\beta,\mu} =Lip_{\beta ,\mu}^1$.
 Obviously the space $Lip_{\beta,\mu}^p$ is
a special case of the so-called weighted Morrey-Campanato spaces
studied by Garc\'ia-Cuerva \cite{g} and other authors, see for example
\cite{t}, \cite{yy} and \cite{sdh}, and references therein.

Given a cube $Q_0$, we need the following maximal function
 with respect to $Q_0$,
$$M_{Q_0}(f)(x) =\sup_{Q_0\supseteq{Q}\ni{x}}
\frac{1}{|Q|}\int_{Q}|f(y)|dy,
$$
where the supremum is taken over all cubes $Q$ such that
$x\in{Q\subseteq{Q_0}}$.

Our results can be stated as follows.

\begin{theorem}\label{thm.1.1}  
Let $b$ be a locally integrable function, $\mu\in {A_1}$ and $0<\beta<1$.
Suppose that $1<p<n/\beta$ and $1/q=1/p-\beta /n$.
Then the following assertions are equivalent:

{\rm(1)} $b \in Li{p_{\beta ,\mu}}.$

{\rm(2)} $M_b$ is bounded from ${L^p}(\mu )$ to ${L^q}(\mu^{1-q})$.
\end{theorem}

\begin{theorem}\label{thm.1.2}   
Let $b$ be a locally integrable function, $\mu\in {A_1}$ and $0<\beta<1$.
Suppose that $1<p<n/\beta$ and $1/q=1/p-\beta /n$.
Then the following assertions are equivalent:

{\rm(1)} $b \in Li{p_{\beta ,\mu }}$ and $b\ge 0$ a.e. in $\rn$.

{\rm(2)} $[b,M]$ is bounded from ${L^p}(\mu )$ to ${L^q}(\mu^{1-q})$.

{\rm(3)} There exists a constant $C>0$ such that
\begin{equation}  \label{equ.thm.2-0}  
\sup_{Q} \frac{1}{\mu(Q)^{\beta/n}}
 \bigg(\frac{1}{\mu(Q)}\int_Q |b(x)-M_Q(b)(x)|^q \mu(x)^{1-q} dx \bigg)^{1/q} \le{C}.
\end{equation}
\end{theorem}

\begin{remark}
Theorems \ref{thm.1.1} and \ref{thm.1.2} give some new characterizations
for centein weighted Lipschitz functions.
 The relevant unweighted results were established in \cite{zp2017}.
\end{remark}

Theorem \ref{thm.1.2} leads to some characterizations for
nonnegative weighted Lipschitz functions.

\begin{corollary}  \label{cor.1} 
Let $0<\beta<1$, $b$ be a locally integrable function and $\mu\in {A_1}$.
Then the following assertions are equivalent:

{\rm(1)} $b \in Li{p_{\beta ,\mu }}$ and $b\ge 0$ a.e. in $\rn$.

{\rm(2)} There exists $1\le{s}<\infty$ such that
\begin{equation}   \label{equ.cor.1}
\sup_Q \frac{1}{\mu(Q)^{\beta/n}}
 \bigg(\frac{1}{\mu(Q)}\int_Q |b(x)-M_Q(b)(x)|^s \mu(x)^{1-s} dx \bigg)^{1/s} \le{C}.
\end{equation}

{\rm(3)} For all $1\le{s}<\infty$, (\ref{equ.cor.1}) holds.
\end{corollary}

For the nonlinear commutator of the sharp function,
 we have the following result.

\begin{theorem}   \label{thm.1.3}   
Let $b$ be a locally integrable function, $\mu\in {A_1}$ and $0<\beta<1$.
Suppose that $1<p<n/\beta$ and $1/q=1/p-\beta /n$.
Then the following assertions are equivalent:

{\rm(1)} $b \in Li{p_{\beta ,\mu }}$ and $b\ge 0$ a.e. in $\rn$.

{\rm(2)} $[b,M^{\sharp}]$ is bounded from ${L^p}(\mu )$ to ${L^q}(\mu^{1-q})$.

{\rm(3)} There exists a constant $C>0$ such that
$$\sup_Q \frac{1}{\mu(Q)^{\beta/n}}
 \bigg(\frac{1}{\mu(Q)}\int_Q |b(x)-2M^{\sharp} (b\chi_Q)(x)|^q \mu(x)^{1-q} dx \bigg)^{1/q} \le{C}.
$$
\end{theorem}

\begin{remark}
The results similar to Theorem \ref{thm.1.3} for $\mu\equiv 1$ were first obtained
in \cite{zws} Corollary 1.2. In the context of variable spaces were
 considered in \cite{zp2019}, see also \cite{zsw}. Similar results were given in Orlics spaces in \cite{zws}.
\end{remark}

Theorem \ref{thm.1.3} also gives some characterizations for
nonnegative weighted Lipschitz functions.

\begin{corollary}  \label{cor.2} 
Let $0<\beta<1$, $b$ be a locally integrable function and $\mu\in {A_1}$.
Then the following assertions are equivalent:

{\rm(1)} $b \in Li{p_{\beta ,\mu }}$ and $b\ge 0$ a.e. in $\rn$.

{\rm(2)} There exists $1\le{s}<\infty$ such that
\begin{equation}   \label{equ.cor.2}
 \sup_Q \frac{1}{\mu(Q)^{\beta/n}}
  \bigg(\frac{1}{\mu(Q)}\int_Q |b(x)-2M^{\sharp} (b\chi_Q)(x)|^s \mu(x)^{1-s} dx \bigg)^{1/s} \le{C}.
\end{equation}

{\rm(3)} For all $1\le{s}<\infty$, (\ref{equ.cor.2}) holds.
\end{corollary}

The rest of this paper is organized as follows.
 In section 2, we give some preliminary facts.
  Section 3 is devoted to proving Theorems 1.1, 1.2  and Corollary 1.1.
   The proof of Theorem 1.3 and Corollary 1.2 will be given in the last section.

\section{Preliminaries and Lemmas}

We will give some preliminaries and lemmas.
 For $x_0\in\rn$ and $r>0$, let $B(x_0,r)$ denote
   the ball centered at $x_0$ with radius $r$.
The following characterizations for the weighted
Lipschitz spaces hold. See Garc\'ia-Cuerva \cite{g} and Tang \cite{t}. 

\begin{lemma}[\cite{t}]     \label{lem.t}  
Let $0<\beta<1$, $w\in{A_1}$ and $1\le{p}\le\infty$.
Then $Lip_{\beta,\mu}^p =Lip_{\beta ,\mu}$ with equivalent norms.
Furthermore, $f\in {Lip_{\beta,w}^p}$
if and only if there exists a constant $C$ such that
$$|f(x)-f(y)|\le {C}\|f\|_{Lip_{\beta,w}}\big[w(B(x,|x-y|))\big]^{\beta/n}(w(x)+w(y))
$$
for a.e. $x, y \in \rn$, especially for $x,y$ being Lebesgue points of $f$.
\end{lemma}

\begin{lemma}  \label{lem.wlip}  
Let $w\in{A_1}$, $0<\beta<1$ and $b\in {Lip_{\beta,w}}$.
 Then there is a constant $C>0$ such that for any cubes $Q\subset\rn$
  and almost all $x\in{Q}$, we have
\begin{equation}     \label{equ.lem.wlip}
|b(x)-b_Q|\le {C}\|b\|_{Lip_{\beta,w}} w(Q)^{\beta/n}w(x).  
 \end{equation}
\end{lemma}

\begin{proof}
Given any cube $Q$, let $B$ be the smallest ball that contains it.
Then ${B}\subset {\sqrt{n}Q}$.
 For any $x, y\in{Q}$, obviously $B(x,|x-y|)\subset 3B$.
   By the doubling property of $A_1$ weights, one has
 $$w(B(x,|x-y|))\le w(3B) \le w(3\sqrt{n}Q)\le {(3\sqrt{n})^n [w]_{A_1}}w(Q),
  $$
where $[w]_{A_1}$ is the $A_1$ constant of $w$ (see \cite{gra} page 504).

Now let $x\in {Q}$ be a Lebesgue point of $b$.
 Applying Lemma \ref{lem.t}  we have
\begin{align*}
|b(x)-b_Q| & \le \frac1{|Q|}\int_Q |b(x)-b(y)|dy\\
  & \le  {C}\|b\|_{Lip_{\beta,w}} \frac1{|Q|}\int_Q \big[w(B(x,|x-y|))\big]^{\beta/n}  (w(x)+w(y))dy\\
   & \le  {C}\|b\|_{Lip_{\beta,w}}w(Q)^{\beta/n}\bigg(w(x) +\frac1{|Q|}\int_Q w(y)dy\bigg)\\
     & \le  {C}\|b\|_{Lip_{\beta,w}}w(Q)^{\beta/n}\big(w(x)+M(w)(x)\big),
\end{align*}
where the constant $C$ above is independent of $Q$ and $b$.

Observing that almost every point in $Q$ is the Lebesgue point of $b$
 and satisfies $M(w)(z)\le [w]_{A_1}w(z)$, then we conclude that
  (\ref{equ.lem.wlip}) holds for a.e. $x\in{Q}$.
 This completes the proof.
\end{proof}

\begin{remark}     \label{rem.lem.wlip}
The proof of Lemma \ref{lem.wlip} shows that if $x$ is
 a Lebesgue point of $b$ and satisfies
 $M(w)(x)\le {[w]_{A_1}}w(x)$, then there is a constant $C>0$
  such that, for any cube $Q\ni {x}$,  
 $$|b(x)-b_Q|\le {C}\|b\|_{Lip_{\beta,w}} w(Q)^{\beta/n}w(x).
 $$
\end{remark}

We also need the following weighted version of
the fractional maximal function.

\begin{definition}\label{def.wfrac}
Let $0<\alpha<n$, $0<r<\infty$ and $w\in {A_{\infty}}$,
the weighted fractional maximal function is defined by
$$M_{\alpha ,\mu ,r}(f)(x)=
  \sup_{Q\ni {x}}\bigg(\frac{1}{\mu(Q)^{1-r\alpha/n}}\int_Q|f(y)|^r \mu(y)dy\bigg)^{1/r}.
$$
\end{definition}

\begin{lemma}[\cite{lmw}]  \label{lem.wfrac}   
Suppose that $0<\alpha<n$, $0<r<p<n/\alpha$ and $1/q=1/p-\alpha/n$.
If $w\in {A_{\infty}}$, then
$$\|M_{\alpha,w,r}(f)\|_{L^q(w)}\le {C}\|f\|_{L^p(w)}.
$$
\end{lemma}

\begin{lemma}[\cite{lmw}]  \label{lem.lmw}   
Let $0<\beta<1$, $w \in {A_1}$ and $b\in {Lip_{\beta,w}}$.
If $f$ is locally integrable and $x\in\rn$,
 then there is a constant $C$ such that
   for any $1<r<\infty$ and any cube $Q\ni{x}$,
$$ \frac1{|Q|}\int_Q |f(y)|dy\le {C}
      w(Q)^{-\beta/n}M_{\beta,w,r}(f)(x)
$$
 and
$$ \frac1{|Q|}\int_Q |b(y)-b_Q||f(y)|dy
  \le {C}\|b\|_{Lip_{\beta,w}}M(w)(x)M_{\beta,w,r}(f)(x).
$$
\end{lemma}

We remark that the first inequality is exactly (2.14) in \cite{lmw}.
 The second one is essentially proved in \cite{lmw} Lemma 2.6(iii)
  except that we replace $w(x)$ by $M(w)(x)$ in the right-hand side.

We now explain how the maximal commutator $M_b$
is controlled by the weighted fractional maximal function
when the symbol $b$ belongs to weighted Lipschitz space.

\begin{lemma} \label{lem.Mb}   
Let $0<\beta<1$, $w \in {A_1}$ and $b\in {Lip_{\beta,w}}$.
Then for any locally integrable function $f$ and any $1<r<\infty$,
$$M_b(f)(x)\le{C}\|b\|_{Lip_{\beta,w}}w(x)M_{\beta,w,r}(f)(x)~ \mbox{a.e.}~ x\in\rn.
$$
\end{lemma}

\begin{proof}
Let $x$ be a Lebesgue point of $b$ and satisfy
 $M(w)(x)\le {[w]_{A_1}}w(x)$. Then,
  the proof of Lemma \ref{lem.wlip} shows that for any cube $Q\ni {x}$
   the following inequality holds
$$|b(x)-b_Q|\le {C}\|b\|_{Lip_{\beta,w}} w(Q)^{\beta/n}w(x).
$$
This together with Lemma \ref{lem.lmw} gives
\begin{align*}
\frac1{|Q|}\int_Q |b(x)-b(y)||f(y)|dy
 & \le |b(x)-b_Q|  \frac{1}{|Q|}\int_Q |f(y)|dy\\
  &\qquad +\frac1{|Q|}\int_Q |b(y)-b_Q||f(y)|dy\\
   & \le {C}\|f\|_{Lip_{\beta,w}}w(x)M_{\beta,w,r}(f)(x).
\end{align*}

Therefore, we deduce that
$$M_b(f)(x)\le{C}\|b\|_{Lip_{\beta,w}}w(x)M_{\beta,w,r}(f)(x), ~\mbox{a.e.}~x\in\rn.
$$
Then we conclude the proof.
\end{proof}

The following two lemmas show that the maximal commutator
 $M_b$ pointwise controls
 the nonlinear commutators $[b,M]$ and  $[b,M^{\sharp}]$.
  See (3.1) and (3.2) in \cite{zw-mia}, respectively.

\begin{lemma} \label{lem.thm1.2-zw-mia}  
Let $b$ and $f$ be locally integrable functions
 and $b\ge0$ a.e. in $\rn$.
Then for any fixed $x\in\rn$ such that $M(f)(x)<\infty$
 and $|b(x)|<\infty$,
we have
\begin{equation}  \label{equ.thm1.2-zw-mia}
 |[b,M](f)(x)|\le M_b(f)(x).
\end{equation}
\end{lemma}

\begin{lemma} \label{lem.thm1.3-zw-mia}  
Let $b$ and $f$ be locally integrable functions
 and $b\ge0$ a.e. in $\rn$.
Then for any fixed $x\in\rn$ such that $M^{\sharp}(f)(x)<\infty$
 and $|b(x)|<\infty$,
we have
\begin{equation}  \label{equ.thm1.3-zw-mia}
|[b,M^{\sharp}](f)(x)|\le {2} M_b(f)(x).
\end{equation}
\end{lemma}

\begin{remark}        \label{rem.lem.2.6-2.7}
Observe that locally integrable function is finite a.e. in $\rn$.
Then (\ref{equ.thm1.2-zw-mia}) holds for a.e. $x\in\rn$ when $M(f)$ is finite a.e. in $\rn$.
 So does (\ref{equ.thm1.3-zw-mia}) when $M^{\sharp}(f)$ is finite a.e. in $\rn$.
\end{remark}

\section{Proof of Theorems \ref{thm.1.1} and \ref{thm.1.2}}  

\begin{proof}[Proof of Theorem \ref{thm.1.1}] 
$(1)\Longrightarrow(2)$ For any $1<r<p<n/\beta$,
  it follows from Lemma \ref{lem.Mb} and Lemma \ref{lem.wfrac} that
\begin{align*}
\|M_b(f)\|_{L^q(\mu^{1-q})}
  & \le {C}\|b\|_{Lip_{\beta,\mu}}\|\mu{M_{\beta,\mu,r}}(f)\|_{L^q(\mu^{1-q})}\\
    &= {C}\|b\|_{Lip_{\beta,\mu}}\|{M_{\beta,\mu,r}}(f)\|_{L^q(\mu)}\\
     &\le {C}\|b\|_{Lip_{\beta,\mu}}\|f\|_{L^p(\mu)}.
\end{align*}

$(2)\Longrightarrow(1)$
 For any cube $Q\subset {{\mathbb{R}}^n}$, we have
\begin{equation}        \label{equ.thm.1-1}
\begin{split}
\frac{1}{\mu(Q)^{1+\beta/n}} \int_Q |b(x)-b_Q| dx
     &\leq \frac{1}{\mu(Q)^{1+\beta/n}}
        \int_Q \bigg(\frac{1}{|Q|}\int_Q |b(x)-b(y)|dy \bigg)dx\\
     &= \frac{1}{\mu(Q)^{1+\beta/n}}
            \int_Q \bigg(\frac{1}{|Q|}\int_Q |b(x)-b(y)|\chi_Q(y)dy\bigg)dx\\
     &\leq \frac{1}{\mu (Q)^{1+\beta/n}} \int_Q {M_b}(\chi_Q)(x) dx.
\end{split}
\end{equation}

Applying H\"older's inequality with exponents $q$ and $q'=q/(1-q)$
and assertion (2), and noticing that $1/p-1/q=\beta/n$, we deduce that
\begin{align*}
 \int_B {M_b}(\chi_Q)(x) dx
  & = \int_Q {{M_b} (\chi_Q)(x)} \mu(x)^{(1-q)/q} \mu(x)^{(q-1)/q}dx\\
   &\le \bigg(\int_Q |M_b(\chi _Q)(x)|^q\mu(x)^{1-q}dx\bigg)^
           {1/q}\bigg(\int_Q \mu(x)dx\bigg)^{1-1/q}\\
    &\le {C}\|M_b(\chi_Q)\|_{{L^q} (\mu^{1-q})} \mu(Q)^{1-1/q}\\
     &\leq{C} \|\chi_Q\|_{{L^p}(\mu)} \mu(Q)^{1-{1}/{q}}\\
      &\le {C} \mu(Q)^{1/p+1-1/q} \\
        &\le {C} \mu(Q)^{1+\beta/n}.
\end{align*}
This together with (\ref{equ.thm.1-1}) yields
$$\frac{1}{\mu(Q)^{1+\beta/n}} \int_Q |b(x)-b_Q| dx\le{C}.
$$
Thus we have $b\in {Lip_{\beta,\mu}}$
since the constant $C$ above is independent of $Q$.

The proof of  Theorem \ref{thm.1.1} is finished.
\end{proof}

\begin{proof}[Proof of Theorem \ref{thm.1.2}]  
$(1)\Longrightarrow(2)$
 Observe that the weighted norm inequalities for
  the Hardy-Littlewood maximal function guarantee
   that $M(f)(x)$ is finite a.e. in $\rn$ when $f\in L^p(\mu)$ and $\mu\in A_1$.
Then, it follows from Lemma \ref{lem.thm1.2-zw-mia}
 and Remark \ref{rem.lem.2.6-2.7} that
  $$|[b,M](f)(x)|\leq M_b(f)(x)~  \mbox{~a.e.}~x\in\rn.
  $$

Therefore, by Theorem \ref{thm.1.1} we get that
 $[b,M]$ is bounded from ${L^p}(\mu)$ into ${L^q}(\mu^{1-q})$.


$(2)\Longrightarrow(3)$  For any fixed $Q \subset\rn$ and $x \in{Q}$,
we have (see \cite{bmr} page 3331 or (2.4) in \cite{zw2009})
$$M (\chi_Q)(x) = {\chi_Q}(x) \hbox{~~and~~} M (b \chi_Q)(x) = {M_Q}(b)(x).
$$

Noting that $1/q=1/p-\beta/n$ and applying assertion (2), we obtain
\begin{align*}
 & \frac{1}{\mu (Q)^{\beta /n}}
      \bigg(\frac{1}{\mu(Q)}\int_Q |{b(x)-M_Q(b)(x)|^q \mu(x)^{1-q}dx \bigg)^{1/q}}\\
  & = \frac{1}{\mu(Q)^{1/p}}
       \bigg(\int_Q |b(x)M (\chi_Q)(x)-{M}(b\chi_Q)(x)|^q {\mu(x)}^{1-q}dx\bigg)^{1/q}\\
    & =  \frac{1}{\mu(Q)^{1/p}} \bigg(\int_Q |[b,M](\chi_Q )|^q {\mu(x)}^{1-q}dx\bigg)^{1/q}\\
     & \leq \frac{1}{\mu(Q)^{1/p}} \|[b,M](\chi_Q)\|_{L^q({\mu}^{1-q})}\\
      & \leq C\mu(Q)^{-1/p}\|\chi_Q\|_{L^p(\mu)}\\
       &\leq C,
\end{align*}
which implies assertion (3) since the constant $C$ is independent of $Q$.

$(3)\Longrightarrow(1)$
 We first prove that assertion (3) implies $b \in {Li{p_{\beta ,\mu}}}$.
  Given any cube $Q$, let
 $$E =\{x \in {Q}: b(x) \leq {b_Q}\}~~\hbox{and}~~ F =\{x \in {Q}: b(x) > {b_Q}\}.
 $$
 Then (see \cite{bmr} page 3331)
 $$\int_E |b(x)-{b_Q}|dx = \int_F |b(x)-{b_Q}| dx.
 $$

Since for any $x \in{E}$ we have $b(x)\leq {b_Q} \leq {M_Q}(b)(x)$,
 then
$$|b(x)-b_Q|\leq |b(x) - {M_Q}(b)(x)|, ~ x\in{E}.
$$
Therefore,
\begin{equation}   \label{equ.thm.2-2}
\begin{split}
\frac{1}{\mu(Q)^{1+\beta/n}}\int_Q {|b(x)-b_Q|} dx
 &=\frac{1}{\mu(Q)^{1+\beta/n}}\int_{E\cup F} |b(x)-b_Q| dx\\
  & =\frac{2}{{\mu(Q)}^{1+\beta/n}} \int_E {|b(x)-b_Q|} dx\\
   & \leq \frac{2}{{\mu(Q)}^{1+\beta/n}} \int_E {|b(x)-{M_Q}(b)(x)|} dx\\
    & \leq \frac{2}{{\mu(Q)}^{1+\beta/n}}\int_Q {|b(x)-{M_Q}(b)(x)|} dx.
\end{split}
\end{equation}

Applying H\"older's inequality and making use of assertion (3), we have
\begin{equation}   \label{equ.thm.2-3}
\begin{split}
& \frac{1}{\mu(Q)^{1+\beta/n}} \int_Q {|b(x)-M_Q(b)(x)|} dx\\
 & \leq \frac{1}{{\mu(Q)}^{1+\beta/n}}
    \bigg(\int_Q |b(x)-{M_Q}(b)(x)|^q{\mu(x)}^{1-q} dx\bigg)^{1/q}
      \bigg(\int_Q \mu(x) dx\bigg)^{1-1/q}\\
   & = \frac{1}{{\mu(Q)}^{\beta/n}}
       \bigg(\frac{1}{\mu(Q)} \int_Q |b(x)-{M_Q}(b)(x)|^q{\mu(x)}^{1-q} dx\bigg)^{1/q}\\
    & \le {C}.
\end{split}
\end{equation}

Combining (\ref{equ.thm.2-2}) and (\ref{equ.thm.2-3}), we deduce that
there is a constant $C>0$ independent of $Q$ such that
$$\frac{1}{\mu(Q)^{1+\beta/n}}\int_Q {|b(x)-b_Q|} dx\le {C},
$$
and hence $b\in Li{p_{\beta,\mu}}$.

Next we prove $b\geq0$ for $a.e.~x\in \rn$.
 Denote by $b^{+}= \max\{b,0\}$ and $b^{-}=|b|-{b^{+}}$, then
  $b = b^{+} - b^{-}$.
It suffices to verify $b^{-} = 0$ for $a.e.~x\in \rn$.

For any cube $Q$, noticing that
  $0 \leq {b^+}(x)\leq |b(x)| \leq {M_Q}(b)(x),$
 we have, for $a.e.~x\in{Q}$,
$$0 \leq {b^{-}}(x)\leq {M_Q}(b)(x)-{b^{+}}(x) + {b^{-}}(x)
    = {M_Q}(b)(x)-b(x).
$$
This, together with (\ref{equ.thm.2-3}), gives
\begin{align*}
 \frac{1}{\mu(Q)^{1+\beta/n}}\int_Q {b^-}(x) dx
   \leq \frac{1}{\mu(Q)^{1+\beta/n}}\int_Q |M_Q(b)(x)-b(x)| dx
     \leq C.
\end{align*}

Therefore we have
\begin{equation}     \label{equ.thm.2-4}
\begin{split}
 \frac{1}{|Q|}\int_Q {b^-}(x) dx
  &\leq C|Q|^{\beta/n}\bigg(\frac{\mu {(Q)}}{|Q|}\bigg)^{1+\beta/n}\\
   &= C|Q|^{\beta/n}\bigg(\frac{1}{|Q|}\int_{Q}\mu(x)dx\bigg)^{1+\beta/n}.
\end{split}
\end{equation}

Now, we show that (\ref{equ.thm.2-4}) implies
 $b^{-}=0$ a.e. in $\rn$.
Indeed, if we denote by ${\mathcal{L}}(b)$ and $\mathcal{L}(\mu)$
the sets of all Lebesgue points of $b$ and $\mu$,
 then $\rn\setminus\big({\mathcal{L}}(b)\cap \mathcal{L}(\mu)\big)$
is a set of measure zero.

For any $x_0\in {\mathcal{L}}(b)\cap \mathcal{L}(\mu)$,
 it follows from Lebesgue's differentiation theorem that
$$\lim_{\substack{|Q|\to 0\\ Q \ni {x_0}}}
    \frac{1}{|Q|}\int_Q b^{-}(x)dx  = b^{-}(x_0)
$$
and
$$\lim_{\substack{|Q|\to 0\\ Q \ni {x_0}}}
      \frac{1}{|Q|}\int_Q \mu(x)dx = \mu(x_0).
$$

Letting $|Q| \to 0$ with $Q \ni {x_0}$ in (\ref{equ.thm.2-4}),
then the left-hand side term tends to $b^{-}(x_0)$ and
 the right-hand side term tends to zero.
 This concludes that $b^-(x_0)=0$.
  Therefore we achieve $b^{-}(x)=0$ for $a. e.~x\in \rn$,
  which gives $b\ge 0$ a.e. in $\rn$.

The proof of Theorem  \ref{thm.1.2} is completed.
\end{proof}

\begin{remark}  \label{rem.thm.1.2} 
Observe that
 the restriction $q\ge n/(n-\beta)$ is not needed
  in the proof of the implication $(3)\Rightarrow(1)$,
although it is implied in the relationships $1<p<n/\beta$
 and $1/q=1/p-\beta/n$.
Therefore, we actually obtained the fact that
 if (\ref{equ.thm.2-0}) holds for some $q\ge1$
     then we also have $b \in Li{p_{\beta ,\mu }}$ and $b\ge 0$ a.e. in $\rn$.
\end{remark}

\begin{proof}[Proof of Corollary \ref{cor.1}]  
The implication $(3)\Rightarrow (2)$ follows readily
and $(2)\Rightarrow (1)$ is implied in Remark \ref{rem.thm.1.2}.
 Therefore, we need to check the implication $(1)\Rightarrow(3)$.
  Suppose that assertion (1) holds.
  We will prove (\ref{equ.cor.1}) for all $1\le{s}<\infty$.

 If $n/(n-\beta)<s<\infty$, we let $r$ be such that
 $1/s=1/r-\beta/n$. Obviously $1<r<n/\beta$.
  Thus, assertion (1) together with Theorem \ref{thm.1.2} (3)
   show that (\ref{equ.cor.1}) holds for $n/(n-\beta)<s<\infty$.

 If $1\le {s}\le n/(n-\beta)$, we select $q$ such that $n/(n-\beta)<q<\infty$.
   Then, it follows from Theorem \ref{thm.1.2} that (\ref{equ.cor.1}) holds for this $q$.
 Now, for any fixed cube $Q\subset\rn$,
 we apply H\"older's inequality with $q/s$ and its conjugate exponent $(q/s)'$
  to deduce
\begin{equation*}
\begin{split}
& \frac{1}{\mu(Q)^{\beta/n}}
 \bigg(\frac{1}{\mu(Q)}\int_Q |b(x)-M_Q(b)(x)|^s \mu(x)^{1-s} dx \bigg)^{1/s}\\
 &= \frac{1}{\mu(Q)^{\beta/n}}
 \bigg(\frac{1}{\mu(Q)}\int_Q |b(x)-M_Q(b)(x)|^s \mu(x)^{(1-q)s/q} \mu(x)^{(q-s)/q} dx \bigg)^{1/s}\\
  &\le \frac{1}{\mu(Q)^{\beta/n}}
   \bigg(\frac{1}{\mu(Q)}\int_Q |b(x)-M_Q(b)(x)|^q \mu(x)^{1-q} dx \bigg)^{1/q}\\
  &\le {C},
\end{split}
\end{equation*}
 which shows that (\ref{equ.cor.1}) also holds when $1\le {s}\le n/(n-\beta)$,
since $C$ is independent of $Q$.

Then the required conclusion follows.
\end{proof}

\section{Proof of Theorem \ref{thm.1.3}}   

\begin{proof}[Proof of Theorem \ref{thm.1.3}]  
$(1)\Longrightarrow(2)$ Noting the fact that
 $M^{\sharp}(f)(x)<\infty$ a.e. $x\in\rn$ when $f\in L^p(\mu)$
and $\mu\in {A_1}$,
then it follows from Lemma \ref{lem.thm1.3-zw-mia},
 Remark \ref{rem.lem.2.6-2.7} and Theorem \ref{thm.1.1} that
$$ \|[b,M^{\sharp}](f)\|_{L^{q}(\mu^{1-q})}
 \le {C}\|M_b(f)\|_{L^{q}(\mu^{1-q})} \le {C}\|f\|_{L^{p}(\mu)}.
$$

$(2)\Longrightarrow(3)$
For any fixed cube $Q$, we have (see \cite{bmr} page 3333  or \cite{zw-mia} page 1383)
$$ M^{\sharp}(\chi_Q)(x) =\frac{1}{2},  ~~\mbox{for~all}~ x\in{Q}.
$$
  Therefore, by assertion (2) and noting that $1/q=1/p-\beta/n$, we obtain
\begin{align*}
&\frac{1}{\mu(Q)^{\beta/n}} \bigg(\frac{1}{\mu(Q)}
       \int_Q {|b(x)-2{M^{\sharp}}(b\chi_Q)(x)|^q \mu(x)^{1-q} dx}\bigg)^{1/q} \\
 &= \frac{2}{\mu(Q)^{\beta/n+1/q}}
    \bigg(\int_Q {\bigg|\frac{1}{2}b(x)-{M^{\sharp}}(b\chi_Q)(x)\bigg|^q \mu(x)^{1-q} dx}\bigg)^{1/q} \\
  &=\frac{2}{\mu(Q)^{1/p}} \bigg(\int_Q {|b(x){M^{\sharp}}(\chi_Q)(x)
          -{M^{\sharp}}(b\chi_Q)(x)|^q \mu(x)^{1-q} dx}\bigg)^{1/q} \\
   &\le \frac{2}{\mu(Q)^{1/p}}
              \|[b,M^{\sharp}](\chi_Q)\|_{L^q({\mu}^{1 - q})}\\
    &\le \frac{C}{\mu(Q)^{1/p}} \|\chi_Q\|_{L^p(\mu)}\\
     &\le {C}.
\end{align*}
Since the constant $C$ is independent of $Q$,
this concludes the proof that (2) implies (3).

(3)$\Longrightarrow$(1). We first proof $b \in {Li{p_{\beta ,\mu}}}$.
 For any cube $Q\subset\rn$, we have (see (2) in \cite{bmr})
\begin{equation}        \label{equ.thm.3-1}
|b_Q|\leq 2M^{\sharp}(b\chi_Q)(x), ~x\in Q.
\end{equation}

Let $E=\{x\in Q: b(x)\le b_B\}$, then
$$\int_E |b(x)-b_Q|dx =\int_{Q\setminus{E}} |b(x)-b_Q|dx.
$$

Since for any $ x\in E$, we have $b(x)\le b_Q\le |b_Q| \le 2
M^{\sharp}(b\chi_Q)(x)$, then
$$|b(x)- b_Q|\le |b(x)-2  M^{\sharp}(b\chi_Q)(x)|,
~~\hbox{for}~ x\in {E}.
$$

Similar to (\ref{equ.thm.2-2}) and (\ref{equ.thm.2-3}),
 applying H\"older's inequality and assertion (3),  we have
\begin{align*} 
& \frac{1}{\mu(Q)^{1+\beta/n}} \int_{Q}|b(x)-b_Q|dx\\
 &= \frac{2}{\mu(Q)^{1+\beta/n}} \int_E |b(x)-b_Q|dx\\
 &\le \frac{2}{\mu(Q)^{1+\beta/n}} \int_{E}|b(x)-2M^{\sharp}(b\chi_Q)(x)|dx\\
  &\le \frac{2}{\mu(Q)^{1+\beta/n}} \int_{Q} |b(x)-2M^{\sharp}(b\chi_Q)(x)|dx\\
   & \leq \frac{1}{{\mu(Q )}^{1+\beta/n}}
    \bigg(\int_Q |b(x)-2M^{\sharp}(b\chi_Q)(x)|^q{\mu(x)}^{1-q} dx\bigg)^{1/q}
      \bigg(\int_Q \mu(x) dx\bigg)^{1-1/q}\\
   & = \frac{1}{{\mu(Q)}^{\beta/n}} \bigg(\frac{1}{\mu(Q)}
        \int_Q |b(x)-2M^{\sharp}(b\chi_Q)(x)|^q{\mu(x)}^{1-q} dx\bigg)^{1/q}\\
    & \le {C},
\end{align*}
which implies $b\in{Lip_{\beta,\mu}}$ since $C$ is independent of $Q$.

Now, let us prove $b\ge 0$ a.e. in $\rn$.
 It also suffices to show $b^{-}=0$ a.e. in $\rn$.
By (\ref{equ.thm.3-1}) we have, for $x\in Q$,
$$2M^{\sharp}(b\chi_Q)(x)-b(x) \ge |b_Q|-b(x)
 = |b_Q|-b^{+}(x)+b^{-}(x).
$$
Then
\begin{equation}         \label{equ.thm.3-2}  
\begin{split}
\frac{1}{|Q|} \int_{Q} \big|2 M^{\sharp} (b\chi_Q)(x)-b(x)\big|dx
&\ge \frac{1}{|Q|}\int_{Q} \big(2 M^{\sharp}(b\chi_Q)(x)-b(x)\big)dx \\
&\ge \frac{1}{|Q|}\int_{Q} \big(|b_Q|
  - b^{+}(x)+b^{-}(x)\big)dx  \\
&= |b_Q| -\frac{1}{|Q|}\int_{Q} b^{+}(x)dx
  + \frac{1}{|Q|}\int_{Q} b^{-}(x)dx.
\end{split}
\end{equation}

On the other hand, by H\"older's inequality and
 assertion (3), we have
\begin{equation*} 
\begin{split}
&\frac{1}{|Q|}\int_Q |2M^{\sharp}(b\chi_Q)(x)-b(x)| dx\\
 &=\frac{1}{|Q|} \int_Q |2M^{\sharp}(b\chi_Q)(x)-b(x)|
    \mu(x)^{(1-q)/q} \mu(x)^{(q-1)/q}dx\\
  &\leq\frac{\mu(Q)}{|Q|}\bigg(\frac{1}{\mu(Q)}
       \int_Q |2M^{\sharp}(b\chi_Q)(x)-b(x)|^q\mu(x)^{1-q} dx\bigg)^{1/q}\\
   & \le {C} \frac{\mu(Q)}{|Q|} \mu(Q)^{\beta/n}\\
     & ={C} |Q|^{\beta/n} \bigg(\frac{1}{|Q|}\int_Q\mu(x)dx \bigg)^{1+\beta/n},
\end{split}
\end{equation*}
where the constant $C$ is independent of $Q$.
 This, together with (\ref{equ.thm.3-2}), gives
\begin{equation}         \label{equ.thm.3-3}
|b_Q| -\frac{1}{|Q|}\int_{Q} b^{+}(x)dx
  + \frac{1}{|Q|}\int_{Q} b^{-}(x)dx
 \le {C} |Q|^{\beta/n} \bigg(\frac{1}{|Q|}\int_Q\mu(x)dx \bigg)^{1+\beta/n}.
\end{equation}

Let $x_0\in \rn$ belong to the intersection
 of the Lebesgue sets of $b, b^{+}, b^{-}$ and $\mu$.
  Letting $|Q|\to0$ with $Q\ni x_0$,
  then it follows form Lebesgue's differentiation theorem that
the limit of the left-hand side of (\ref{equ.thm.3-3}) equals to
$$|b(x_0)|-b^{+}(x_0)+b^{-}(x_0)=2b^{-}(x_0)=2|b^{-}(x_0)|,
$$
and the right-hand side of (\ref{equ.thm.3-3}) tends to $0$.
  Hence, we have $b^{-}(x_0)=0$.

Therefore, the proof of Theorem \ref{thm.1.3} is completed.
\end{proof}

\begin{proof}[Proof of Corollary \ref{cor.2}]  
Arguing in a similar way as the proof of Corollary \ref{cor.1},
we can get Corollary \ref{cor.2}.
We leave the details to the readers.
\end{proof}

\bigskip
{\bf Acknowledgments~}
   This work was partly supported by 
    the Fundamental Research Funds for Education Department of Heilongjiang Province (NO. 1453ZD031).

\medskip
{\bf Data availibility}
This research paper does not involve any data.

{\bf Competing interests~}
  The authors declare that they have no competing interests.

\medskip
{\bf Authors' contributions~}
All authors contributed equally to this work. All authors read
the final manuscript and approved its submission.

\end{document}